\documentclass[a4paper,12pt]{article}

\usepackage{amsmath, wrapfig}
\usepackage{dsfont, a4wide, amsthm, amssymb, amsfonts, graphicx}
\usepackage{fancyhdr, xspace, psfrag, setspace, supertabular, color, enumerate}
\usepackage{hyperref}

\newtheorem{theorem}{Theorem}[section]
\newtheorem{proposition}[theorem]{Proposition}

\newtheorem{corollary}[theorem]{Corollary}
\newtheorem{conjecture}[theorem]{Conjecture}
\newtheorem{definition}[theorem]{Definition}

\newtheorem{example}[theorem]{Example}
\newtheorem{question}[theorem]{Question}
\newtheorem{remark}[theorem]{Remark}
\newtheorem*{example*}{Example}

\def\b{\beta}

\def\e{\epsilon}

\def\ct{\mathcal{T}}
\def\Z{\mathbb{Z}}
\def\F{\mathbb{F}}

\def\E{\mathbb{E}}
\def\P{\mathbb{P}}

\def\N{\mathbb{N}}

\title{On small densities defined without pseudorandomness}
\author{Thomas Karam\footnote{Mathematical Institute, University of Oxford. Email: \texttt{thomas.karam@maths.ox.ac.uk}.}}

\begin{document}
\maketitle

\begin{abstract}

We identify an assumption on linear forms $\phi_1, \dots, \phi_k: \F_p^n \to \F_p$ that is much weaker than approximate joint equidistribution on the Boolean cube $\{0,1\}^n$ and is in a sense almost as weak as linear independence, but which guarantees that every subset of $\{0,1\}^n$ on which none of $\phi_1, \dots, \phi_k$ has full image has a density which tends to 0 with $k$. This density is at most quasipolynomially small in $k$, a bound that is necessarily close to sharp.

\end{abstract}

\tableofcontents

\section{Introduction}

\subsection{Background on mod-$p$ linear forms and basic difficulties}

All our statements in this paper will be uniform in the integer $n$, which will not play any significant role. 

The present paper will involve restrictions of mod-$p$ forms (that is, of linear forms $\F_p^n \to \F_p$) to subsets of $\F_p^n$ such as $\{0,1\}^n$ and more generally $S^n$ for some non-empty subset $S$ of $\F_p$. These restrictions have been studied in several recent works, such as \cite{Gowers and K equidistribution}, \cite{Gowers and K approximation}, \cite{K. ranges} (by Gowers and the author, for the first two) in the more general settings of polynomials and of abelian group homomorphisms, as well as in \cite{K. correlation}, which focused on the question of how many dense subsets of the cube $\{0,1\}^n$ could be built so that any two of them can be strongly distinguished by some mod-$p$ form, in various senses. 

Another context in which these restrictions of mod-$p$ forms appear is that of obstructions to combinatorial statements of Hales-Jewett type, such as those discussed in \cite{Gowers}. We refer the reader to the introductions of \cite{Gowers and K equidistribution}, \cite{Gowers and K approximation}, \cite{K. correlation} for a discussion of that connection.

Restrictions of $\F_p^n$ to $\{0,1,2\}^n$ also arise in the recent work of Bhangale, Khot and Minzer \cite{Bhangale Khot Minzer} where they considerably improve the best known bounds on the size of subsets of $\F_p^n$ which do not contain any length three arithmetic progression $\{x,x+a,x+2a\}$ where the common difference $a$ is furthermore required to be an element of $\{0,1,2\}^n \setminus \{0\}$. As mentioned in their paper, the slightly stronger variant where $\{0,1,2\}$ is replaced by $\{0,1\}$ had been asked by Hazla, Holenstein and Mossel \cite{Hazla Holenstein Mossel} and emphasized by Green \cite{Green}.

In this paper we continue the development of the basic theory of sets defined by restricted mod-$p$ forms, in a new direction that focuses on a perhaps surprisingly weak condition - in particular, much weaker than the condition used in previous works \cite{Gowers and K approximation}, \cite{K. correlation} - which nonetheless suffices to guarantee that a subset of $\{0,1\}^n$ defined by a large conjunction of non-trivial conditions involving mod-$p$ forms is not dense.

The distributions of mod-$p$ forms and $k$-tuples of mod-$p$ forms on $\{0,1\}^n$ for some integer $k$ present additional difficulties compared to their counterparts on the whole of $\F_p^n$. On the latter, we have an equivalence between linear independence and joint equidistribution: a $k$-tuple of mod-$p$ forms is jointly equidistributed (that is, it takes each value of $\F_p^k$ with probability $p^{-k}$) if and only if it is a linearly independent family. On $\{0,1\}^n$, substantially more complicated behaviour arises (provided that $p \ge 3$): some forms such as $x_1$ do not even have full range, so even for $k=1$ there is something to be said: if the mod-$p$ form depends on a large number $r$ of coordinates, then as we will recall in Proposition \ref{equidistribution for k-tuples of mod-p forms} and its proof a Fourier-analytic calculation shows that the form is approximately equidistributed, with an error term decreasing exponentially with $r$. If on the other hand the mod-$p$ form only depends on some $r$ coordinates, then divisibility arguments provide a lower bound on the distance to equidistribution (using that $p \ge 3$): the probability that the form takes any given value is an integer multiple of $2^{-r}$, so its difference with $1/p$ is at least $1/2^rp$ in absolute value.

Moving from one to two mod-$p$ forms $\phi_1$, $\phi_2$, we again do not have approximate joint equidistribution nor even approximate independence of the events \[\phi_1(x) = y, \phi_2(x) = z\] for general linearly independent $\phi_1$, $\phi_2$, and general $y,z$ when $x$ is chosen at random in $\{0,1\}^n$ even if $\phi_1, \phi_2$ both depend on many coordinates: this can be seen by taking $\phi_2 = \phi_1 + x_1$ for an arbitrary choice of $\phi_1$. Rather, Fourier analysis there again shows (by a short calculation, of the kind that we shall do in a moment when proving Proposition \ref{equidistribution for k-tuples of mod-p forms}) that a sufficient condition for $\phi_2(x)$ to be approximately equidistributed conditionally on any value of $\phi_1(x)$ is that $\phi_2 - b \phi_1$ has large support for any $b \in \F_p$, and that a sufficient condition for the pair $(\phi_1(x), \phi_2(x))$ to be approximately equidistributed is that $a \phi_1 + b \phi_2$ has large support for any $(a,b) \in \F_p^2 \setminus \{0\}$. That condition, in turn, extends to $k$-tuples of mod-$p$ forms for any $k \ge 1$, and was already studied in \cite{Gowers and K approximation}. We recall the relevant definition, statement and proof as starting points.

\begin{definition}

Let $p$ be a prime. If $\phi: \F_p^n \to \F_p$ is a mod-$p$ form defined by \[\phi(x) = a_1 x_1 + \dots + a_n x_n,\] then we say that the \emph{support} $Z(\phi)$ is the set \[\{i \in [n]: a_i \neq 0\},\] and that the \emph{support size} of $\phi$ is the size of $Z(\phi)$.

If $k \ge 1$, $r \ge 0$ are integers, then we say that mod-$p$ forms $\phi_1, \dots, \phi_k$ are \emph{$r$-separated} if the support size of the linear combination \[a_1 \phi_1 + \dots + a_k \phi_k\] is at least $r$ for every $(a_1, \dots, a_k) \in \F_p^k \setminus \{0\}$.

\end{definition}

\begin{proposition} [\cite{Gowers and K approximation}, Proposition 2.4]\label{equidistribution for k-tuples of mod-p forms}

Let $p$ be a prime, let $S$ be a subset of $\F_p$ with size at least $2$, and let $k \ge 1$, $r \ge 0$ be integers. If $\phi_1, \dots, \phi_k: \F_p^n \to \F_p$ are $r$-separated mod-$p$ forms, then for every $(y_1, \dots, y_k) \in \F_p^k$ we have \[|\P_{x \in \{0,1\}^n}(\phi_1(x) = y_1, \dots, \phi_k(x) = y_k) - p^{-k}| \le (1-p^{-2})^r.\]

\end{proposition}

\begin{proof}

The probability on the left-hand side can be expressed as \begin{align*} \E_{x \in S^n} 1_{\phi_1(x) = y_1} \dots 1_{\phi_k(x) = y_k}
= & \E_{x \in S^n} \left( (\E_{a_1 \in \F_p} \omega_p^{a_1(\phi_1(x) - y_1)}) \dots (\E_{a_k \in \F_p} \omega_p^{a_k(\phi_k(x) - y_k)}) \right) \\
= & \E_{a \in \F_p^k} F(a) \end{align*} where $F(a)$ is the Fourier coefficient \[\E_{x \in S^n} \omega_p^{a_1(\phi_1(x) - y_1) + \dots + a_k(\phi_k(x) - y_k)}.\] The contributions of each coordinate of $x$ are independent, so for a fixed $a \in \F_p^k$ we may decompose \[F(a) = \omega_p^{-(a_1 y_1 + \dots + a_ky_k)} \prod_{z=1}^n \E_{x_z \in S} \omega_p^{(a_1 \phi_{1,z} + \dots + a_k \phi_{k,z})x_z}\] where $\phi_{i,z} \in \F_p$ is the coefficient of the form $\phi_i$ at the coordinate $z$ for every $i \in [k]$ and every $z \in [n]$. It is not difficult to show (see for instance, \cite[Lemma 2.1]{Gowers and K approximation} and its proof) that whenever $t \in \F_p^*$ we have $|\E_{u \in S} \omega_p^{tu}| \le 1-p^{-2}$. Therefore \[|F(a)| \le (1-p^{-2})^{|Z(a_1 \phi_1 + \dots + a_k \phi_k)|}\] for every $a \in \F_p^k$. Using $F(0) = p^{-k}$ and the assumption, the desired inequality follows. \end{proof}

The reader not familiar with mod-$p$ linear forms restricted to $\{0,1\}^n$ may also want to view the previously described behaviours as analogous to those that arise when considering quadratic forms on the whole of $\F_p^n$. Indeed, for $p \ge 3$ a non-zero quadratic form $\F_p^n \to \F_p$ is not necessarily approximately equidistributed (or even surjective), but if it has sufficiently high rank then it is. Likewise, linearly independent $k$-tuples of quadratic forms $(q_1, \dots, q_k)$ are not approximately equidistributed in general, but if every linear combination \[a_1 q_1 + \dots + a_k q_k\] of them has large enough rank then they are. The first of these facts follows from diagonalising the quadratic form, and a generalisation to polynomials is discussed in a landmark paper \cite{Green and Tao} of Green and Tao. The second fact reduces to the first by a standard Fourier-analytic calculation.

\subsection{Main results}

Previous works \cite{Gowers and K approximation}, \cite{K. correlation} relied primarily on Proposition \ref{equidistribution for k-tuples of mod-p forms} to show that a set defined by linear mod-$p$ conditions has small density. One quantitative way in which Proposition \ref{equidistribution for k-tuples of mod-p forms} may be viewed as not so strong is that in order to immediately deduce from Proposition \ref{equidistribution for k-tuples of mod-p forms} an upper bound of the type \[|\P_{x \in \{0,1\}^n}(\phi_1(x) = y_1, \dots, \phi_k(x) = y_k) - p^{-k}| \le cp^{-k}\] for some fixed $c>0$, the lower bound $r$ on the supports of the linear combinations \[a_1 \phi_1 + \dots + a_k \phi_k\] with $a \in \F_p^{k} \setminus \{0\}$ that we must take grows linearly in $k$.

In particular, if we consider strict subsets $E_1, \dots, E_k$ of $\F_p$ and the set $\Lambda$ of $x \in \{0,1\}^n$ satisfying all conditions \[\phi_1(x) \in E_1, \dots, \phi_k(x) \in E_k,\] as has been done in \cite{Gowers and K approximation}, then it suffices for $E_1, \dots, E_k$ to have size $2$ or more for the upper bound on the density of $\Lambda$ that comes from applying the bounds from Proposition \ref{equidistribution for k-tuples of mod-p forms} to be \[2^k (1-p^{-2})^r\] and for that upper bound to be meaningful for large $k$ we must then require $r$ to be at least linear in $k$.

Our main contribution in the present paper will be to show that there exists some lower bound on $r$ that is \emph{uniform} in $k$, and which suffices to guarantee that the density of $\Lambda$ tends to $0$ as $k$ tends to infinity. We will state and prove our results in the general case where the alphabet of the variables is an arbitrary subset $S$ of $\F_p$ containing at least two elements, as the proofs in this setting are the same as those of the special case $S = \{0,1\}$.

\begin{definition}

Let $p$ be a prime, let $S$ be a subset of $\F_p$ with size at least $2$, and let $E$ be a strict subset of $\F_p$. Then there exists a smallest nonnegative integer \begin{equation} L=L(S,E) \le \lfloor \frac{|E|-1}{|S|-1} \rfloor + 1 \le |E|\label{inequality on L} \end{equation} such that for any $a_1, \dots, a_L \in \F_p^*$ the set \begin{equation} a_1 S + \dots + a_L S \label{generating F_p} \end{equation} is not contained in $E$. We will write $L(S)$ for $L(S,\F_p)$.

\end{definition}

The existence of $L$ follows from the $k=1$ case of Proposition \ref{equidistribution for k-tuples of mod-p forms}, but it also follows immediately from the Cauchy-Davenport inequality, which furthermore provides the inequality \eqref{inequality on L}. We note that if $S = \{0, 1, \dots, |S|-1\}$ and $E = \{0, 1, \dots, |E|-1\}$, then \eqref{inequality on L} becomes an equality. We are now ready to state our main result in this paper.

\begin{theorem} \label{Main result}

Let $p$ be a prime, and let $S$ be a subset of $\F_p$ with size at least $2$. Then there exists some $a=a(p,S)>0$ such that the following holds. If $k \ge 1$ is an integer, $E_1, \dots, E_k$ are strict subsets of $\F_p$, and $\phi_1, \dots, \phi_k$ are mod-$p$ forms satisfying \begin{equation} |Z(\phi_j - \phi_i)| \ge 2 \max_{1 \le t \le k} L(S,E_t) - 1 \label{assumption of main result} \end{equation} for any $1 \le i < j \le k$, then the density of points $x \in S^n$ satisfying \[\phi_1(x) \in E_1, \dots, \phi_k(x) \in E_k\] is at most $O(k^{-a/\log \log k})$ as $k$ tends to infinity.

\end{theorem}

In the case where the sets $E_1, \dots, E_k$ all have size bounded above by some common value strictly less than $p$, Theorem \ref{Main result} specialises as follows thanks to the inequality \eqref{inequality on L}.

\begin{corollary} \label{Corollary of main result}

Let $p$ be a prime, let $S$ be a subset of $\F_p$ with size at least $2$, and let $1 \le r \le p-1$ be an integer. Then there exists some $a=a(p,S,r)>0$ such that the following holds. If $k \ge 1$ is an integer, $E_1, \dots, E_k$ are strict subsets of $\F_p$ each with size at most $r$, and $\phi_1, \dots, \phi_k$ are mod-$p$ forms satisfying \[|Z(\phi_j - \phi_i)| \ge 2 \lfloor \frac{r-1}{|S|-1} \rfloor + 1\] for any $1 \le i < j \le k$, then the density of points $x \in S^n$ satisfying \[\phi_1(x) \in E_1, \dots, \phi_k(x) \in E_k\] is at most $O(k^{-a/ \log \log k})$ as $k$ tends to infinity.

\end{corollary}

The bound in the assumption of Theorem \ref{Main result} is optimal, at least in the case where the sets $E_1, \dots, E_k$ are the same set $E$: indeed, by definition of $L=L(S,E)$ we can find $a_1, \dots, a_{L-1}$ such that \[a_1 S + \dots + a_{L-1} S \subset E.\] Taking $\phi_1, \dots, \phi_k$ to be mod-$p$ forms with pairwise disjoint supports, and with non-zero coefficients equal to $a_1, \dots, a_{L-1}$ (respecting multiplicities), we have that $\phi_i(S^n) \subset E$ for every $i \in [k]$, and that \[|Z(\phi_j - \phi_i)| = |Z(\phi_j)| + |Z(\phi_i)| = 2L-2\] for any $1 \le i < j \le k$.

We are about to use the following proposition, proved in \cite{Gowers and K approximation}, which informally states that the probability that a given mod-$p$ form takes any prescribed value is either $0$ or bounded below by some positive quantity. Proposition \ref{Lower bound on non-zero probabilities} will again be used in Section \ref{Section: Proof of the main result}.

\begin{proposition}[\cite{Gowers and K approximation}, Proposition 2.5]\label{Lower bound on non-zero probabilities}  Let $p$ be a prime, let $S$ be a non-empty subset of $\F_p$, and let $\phi: \F_p^n \rightarrow \F_p$ be a linear form. Then for every $y \in \F_p$ we have \[\P_{x \in S^n}(\phi(x) = y) = 0 \text{ or }\P_{x \in S^n}(\phi(x) = y) \ge \b(p,S) \] where $\b(p,S) = |S|^{-\lceil(p-1)/(|S|-1)\rceil}$.\end{proposition}

We note that the bounds in Theorem \ref{Main result} and in Corollary \ref{Corollary of main result} cannot be too far from tight, because there are simple examples satisfying power lower bounds in the other direction (and furthermore, uniformly with respect to $r \ge 1$). 

\begin{example}\label{Conclusion not too far from optimal bounds in main theorem}

Let $r,t \ge 1$ be integers and let $k = p^t$. Let $\psi_1, \dots, \psi_t: \F_p^n \to \F_p$ be linear forms with pairwise disjoint supports each with size at least $r$. The $k$ linear combinations $\phi_1, \dots, \phi_k$ of $\psi_1, \dots, \psi_t$ then satisfy $|Z(\phi_j - \phi_i)| \ge r$ for any $1 \le i < j \le k$, but the set \[\{x \in \{0,1\}^n: \phi_1(x) = 0, \dots, \phi_k(x) = 0\}\] is the same as the set \[\{x \in \{0,1\}^n: \psi_1(x) = 0, \dots, \psi_t(x) = 0\},\] so by Proposition \ref{Lower bound on non-zero probabilities} has density at least $2^{-(p-1)t} = k^{-a}$ where $a = \log (2^{p-1}) / \log p$.

\end{example}

As Example \ref{Conclusion not too far from optimal bounds in main theorem} illustrates, the assumption that the difference of any two distinct mod-$p$ forms of some family has support size bounded below by some integer is in particular insufficient on its own to guarantee that the joint distribution of the family of forms is anywhere near uniform. Towards the end of Section \ref{Section: Constant separation is much weaker than approximate independence}, we will discuss an example (Example \ref{Example 4}) where a power lower bound on the density can be obtained even for a set defined in terms of conditions where the linear forms satisfy the much stronger condition of being $r$-separated. (However, the bound will not be uniform with respect to large $r$ this time.) This will in turn illustrate that this stronger assumption does not remotely come close to ensuring approximate equidistribution either.

More generally, Section \ref{Section: Constant separation is much weaker than approximate independence} will as a whole focus on $r$-separation for a fixed $r$ and on illustrating the various ways that it comes short of guaranteeing several properties that are guaranteed by the assumption of Proposition \ref{equidistribution for k-tuples of mod-p forms}. By contrast, in Section \ref{Section: Proof of the main result} we shall prove Theorem \ref{Main result}, the assumption of which is qualitatively weaker than separation. Finally, Section \ref{Section: Open questions} will be devoted to some open questions.

Throughout, we will often write $(\phi, E)$ for a condition $\phi(x) \in E$ where $x$ is an element of $S^n$, as was done in \cite{Gowers and K approximation}. If $\phi, \psi$ are two mod-$p$ forms, then we will say that the \emph{support distance} (or more simply \emph{distance}) between $\phi$ and $\psi$ is the support size of $\phi - \psi$. If $\phi_0$ is a mod-$p$ form and $r \ge 0$ is an integer, then we will say that the \emph{ball} with \emph{radius} $r$ centered at $\phi_0$ is the set of mod-$p$ forms $\phi$ such that the support size of $\phi - \phi_0$ is at most $r$.

\section*{Acknowledgements}

The author is grateful to Timothy Gowers for helpful discussions in Spring 2019, when the content of the present paper arose. The author also thanks Noga Alon and Swastik Kopparty for encouraging comments on related projects involving restrictions to $\{0,1\}^n$ of objects defined on $\F_p^n$ and for their optimism about this area. This research received funding from an EPSRC International Doctoral Scholarship (project number 2114524).

\section{Constant separation is much weaker than approximate independence} \label{Section: Constant separation is much weaker than approximate independence}

Throughout this section we fix a prime $p \ge 3$. We will construct examples of situations where linear forms that are $r$-separated for some fixed integer $r$ (or, in Subsection \ref{Subsection: Adding a new condition with small or absent effects}, rather satisfy a variant of this assumption that is focused on the last mod-$p$ form) nonetheless exhibit behaviour that is far from joint independence in the probabilistic sense, in various ways. We will take $S = \{0,1\}$ for simplicity, but the examples can be adapted rather straightforwardly to any non-empty $S \neq \F_p$.

\begin{example} \label{Example 1} \emph{The most basic example is the case $r=1$, that is, the assumption that the linear forms are merely linearly independent. The linear forms $x_1 + x_i$ with $i \neq 1$ are not approximately independent in the probabilistic sense, and that is even more striking if we furthermore add the linear form $x_1$ to the family, since the value of $x_1$ then forces the values of all $x_1 + x_i$ to belong to $\{x_1\} + \{0,1\}$. This example, and examples of sunflower type more generally, were the starting point of the paper \cite{Gowers and K approximation} and have been foundational for its proofs.} \end{example}

\subsection{Adding a new condition with small or absent effects}\label{Subsection: Adding a new condition with small or absent effects}

Let $r \ge 1$ be any fixed integer. In this subsection we focus on the effect of an additional condition, and correspondingly we will not work under the full assumption that the mod-$p$ forms that we consider are $r$-separated, but that, informally speaking, the last condition is $r$-separated from all previous ones. More precisely our setting will be to consider “preexisting” mod-$p$ forms $\phi_1, \dots, \phi_k$, and a “new” mod-$p$ form $\phi$ satisfying \begin{equation}|Z(\phi - \sum_{i=1}^k a_i \phi_i)| \ge r \label{r-separation focused on the last form} \end{equation} for every $a \in \F_p^k$. We begin by an example showing that it is possible to find an integer $k \ge 1$, mod-$p$ forms $\phi_1, \dots, \phi_k$, and a mod-$p$ form $\phi$ satisfying \eqref{r-separation focused on the last form} for every $a \in \F_p^k$ and yet still have that for some subsets $E_1, \dots, E_k,E$ with $E \neq \F_p$ the condition $\phi(x) \in E$ is implied by the conditions $\phi_i(x) \in E_i$ in the sense that every $x \in \{0,1\}^n$ satisfying the latter also satisfies the former.

\begin{example} \emph{Let $q = (p-1)/2$. We take $k = \lceil r/q \rceil$. For each $i \in [k]$ we define \begin{align*} \psi_i & = x_{(2i-2)q+1} + \dots + x_{2(i-1)q}\\
\rho_i & = x_{(2i-1)q+1} + \dots + x_{2iq}\\
\phi_i & = \psi_i + \rho_i \end{align*} and then further define \[\phi = \psi_1 + \dots + \psi_k.\] For every $a \in \F_p^k$ the linear combination \begin{equation} \phi - \sum_{i=1}^k a_i \phi_i \label{combination of which we take support} \end{equation} has support size at least $kq \ge r$. Indeed, letting $J(a) \subset [k]$ be the set \[ \{i \in [k]: a_i \neq 0\},\] the support of the linear combination \eqref{combination of which we take support} contains the disjoint union \[ \bigcup_{i \in J(a)} Z(\rho_i) \bigcup \bigcup_{i \in [k] \setminus J(a)} Z(\psi_i), \] which has size $kq$. But if $x \in \{0,1\}^n$ satisfies $\phi_i(x)=0$ for each $i \in [k]$, then $\psi_i(x) = 0$ for each $i \in [k]$ and hence $\phi(x) = 0$.} \end{example}

We now discuss another example, showing that even if some conditions $(\phi_i, E_i)$ do not imply some condition $(\phi,E)$ in the sense from the previous example, the ratio of the density of the subset \[\{x \in \{0,1\}^n: \phi_1(x) \in E_1, \dots, \phi_k(x) \in E_k, \phi(x) \in E\}\] over the density of the subset \[\{x \in \{0,1\}^n: \phi_1(x) \in E_1, \dots, \phi_k(x) \in E_k\}\] can be arbitrarily close to $1$.

\begin{example}

\emph{Let $k$ be a fixed integer which we will later take to be large enough. We let $\psi_0, \psi_1, \dots, \psi_k$ be the mod-$p$ forms defined by \[\psi_i = x_{ir+1} + \dots + x_{(i+1)r}\] for every $i \in \{0\} \cup [k]$. We then take $\phi_i = \psi_0 + \psi_i$ for each $i \in [k]$ and $\phi = \psi_0$. By definition the linear forms $\psi_1, \dots, \psi_k$ all have the same distribution $D$ on $\{0,1\}^n$. The probability mass that this distribution puts on any element of $\F_p$ is an integer multiple of $2^{-r}$, and that cannot be $1/p$. This shows in particular that $D$ is not perfectly uniform.}

\emph{Our construction is designed to amplify the effect of the lack of uniformity of $D$ on the distribution of $\phi$ given suitable values of $\phi_1, \dots, \phi_k$. We assume for simplicity that $r \ge p-1$, so that \[\P(\phi =y) > 0, D(v) > 0\] for all $y,v \in \F_p$. We have \begin{align*} \P( \phi =y, \phi_1 = u, \dots, \phi_k =u) & = \P( \phi = y) \P(\psi_1 = u-y) \dots \P(\psi_k = u-y) \\ & = \P( \phi = y) D(u-y)^k. \end{align*} for all $y,u \in \F_p$. Let $E$ be the set of $y \in \F_p$ which minimise $D(u-y)$. The ratio of the density of \[ \{ x \in \{0,1\}^n: \phi_1(x) = u, \dots, \phi_k(x) = u, \phi(x) \in E \} \] over the density of \[ \{ x \in \{0,1\}^n: \phi_1(x) = u, \dots, \phi_k(x) = u \}\] then becomes exponentially close to $1$ as $k$ tends to infinity.} \end{example}

\subsection{A density with a power lower bound}

We now return from the effect of one additional condition to what occurs more globally for a system of conditions. In the next example we will show that no integer $r \ge 1$ guarantees that the set of points of $\{0,1\}^n$ satisfying a family of $r$-separated conditions has a density that (uniformly in $n$) is at most exponential in the number of conditions, and that a power lower bound on the density is even possible.

\begin{example} \label{Example 4} \emph{Let $k$ be a fixed integer, let $T=T(r,k)$ be an integer that we will choose later, let $\psi_1, \dots, \psi_k$ be mod-$p$ forms \emph{supported inside} \[\ct=\{k+1, k+2, \dots, k+T(r,k)\}\] which we will also construct later, and let $\phi_i = \psi_i + x_i$ for each $i \in [k]$.}

\emph{For the forms $\phi_1, \dots, \phi_k$ to be $r$-separated, it suffices that every linear combination \[a_1 \phi_1 + \dots + a_k \phi_k\] where at most $r-1$ values $a_1, \dots, a_k$ are non-zero has support size at least $r$, as otherwise that follows immediately from considering the contributions of $x_1, \dots, x_k$ to the support. In turn, for that it suffices that for these $a \in \F_p^k$ the linear combinations \[a_1 \psi_1 + \dots + a_k \psi_k\] have support size at least $r$.}

\emph{To satisfy that we now choose the forms $\psi_1, \dots, \psi_k$ by induction: we begin by choosing some $\psi_1$ with $|Z(\psi_1)| \ge r$, and for any $u < k$, assuming that $\psi_1, \dots, \psi_u$ have been fixed we choose $\psi_{u+1}$ to be some form supported inside $\ct$ and outside the union $U_u$ of all balls with radius $r-1$ centered at any linear combination \[a_1 \psi_1 + \dots + a_{u} \psi_u\] where at most $r-1$ values $a_1, \dots, a_u$ are non-zero. The number of such $a \in \F_p^u$ is at most \begin{equation} \sum_{i=0}^{r-1} p^i \binom{u}{i} \le p^r \binom{k}{r}\label{estimate for binomial coefficients} \end{equation} for $k$ large enough (depending on $r$) and the number of mod-$p$ forms supported inside $\ct$ and with support size at most $r-1$ is at most \[\sum_{j=0}^{r-1} p^j \binom{T}{j} \le p^r \binom{T}{r}\] for $T$ large enough (depending on $r$). The total number of mod-$p$ forms in $U_u$ is hence at most \[p^{2r} \binom{k}{r} \binom{T}{r} \le p^{2r} (kT)^r.\] Meanwhile the number of possible mod-$p$ forms supported inside $\ct$ is $p^T$, so provided that \begin{equation} p^T > p^{2r} (kT)^r \label{inequality at the end of Example 4} \end{equation} we may inductively choose $\psi_1, \dots, \psi_k$ as desired. As for \eqref{inequality at the end of Example 4}, we may (for fixed $p$ and $r$) choose $T$ which satisfies \eqref{inequality at the end of Example 4} and which has a growth rate $T = O(\log k)$ in the value of $k$.}

\emph{If $x \in \{0,1\}^n$ is such that $x_z = 0$ for every $z \in \ct$, then for every $i \in [k]$ we have $\psi_i(x) = 0$ and hence $\phi_i(x) \in \{0,1\}$. The set of points $x \in \{0,1\}^n$ satisfying $\phi_i(x) \in \{0,1\}$ for every $i \in [k]$ therefore has density at least $2^{-T}$, which is at least some negative power of $k$ (that depends on $p$ and $r$). A back-of-envelope calculation shows that for any $\e>0$ this density is at least \[k^{-((1+\e) \log 2 / \log p) r}\]  once $k$ is large enough depending on $\e$.}

\end{example}

As Example \ref{Example 4} shows, $r$-separation for a fixed $r$ is in a way not much stronger than linear independence, since linear independence of a family of $k$ mod-$p$ forms implies that the union of the supports of the $k$ forms has size at least $k$, whereas in this example it has size $k+O(\log k)$. It is hence all the more remarkable that Theorem \ref{Main result} holds with its current assumptions.

\begin{remark}

\emph{The mod-$p$ forms used in Example \ref{Example 4} satisfy an assumption that is much stronger than that of Theorem \ref{Main result}: not only does the support of the difference of any two of these forms have size at least $r$, but the support of any linear combination of them has size at least $r$. It is hence natural to wonder for a moment whether we could not improve the power lower bound still much further and contradict Theorem \ref{Main result}, by considering the linear combinations of all forms involved in Example \ref{Example 4} (as they still satisfy the assumption of Theorem \ref{Main result}). Where this line of thought hits a wall is that if \[\phi = a_1 \phi_1 + \dots + a_k \phi_k\] is a linear combination with at least $p-1$ non-zero coefficients $a_1,\dots,a_k$ then the image of the set \[\{x \in \{0,1\}^n: \forall z \in \ct, x_z=0\}\] by $\phi$ is the full of $\F_p$, as can be seen by considering the contribution of the coordinates inside $[k]$, which is independent from that of the other coordinates. Therefore, we cannot use that linear form to build a condition $(\phi, E)$ with $E \neq \F_p$ that is satisfied by this set. Since the number of other possibilities for $(a_1, \dots, a_k)$ is at most $p^{p} \binom{k}{p-1}$ (by an estimate similar to \eqref{estimate for binomial coefficients}), so at most a fixed power of $k$, the best that can be achieved by taking this enlarged set of forms is still a power lower bound in $k$ (with a power divided by $p-1$ compared to that of Example \ref{Example 4}).}

\end{remark}

\section{Proof of the main result}\label{Section: Proof of the main result}

In this section we prove Theorem \ref{Main result}. Although there are several ways of writing this proof, the choice that we ultimately made was to do so in a way that mirrors the high-level structure of that of \cite[Theorem 1.2]{Gowers and K approximation} (performed in Section 3 of that paper). In both situations the proof proceeds in three successive stages establishing the desired result, in the cases where the linear forms that we start with (i) constitute a sunflower, (ii) are contained in a ball with bounded radius, and finally (iii) are completely arbitrary. Let us begin by recalling the relevant formal definition of a sunflower.

\begin{definition}

Let $p$ be a prime, let $I$ be a finite set, let $\phi_i: \F_p^n \to \F_p$ be a linear form for every $i \in I$, and let $\phi_0$ be an additional such form. We say that the forms $\phi_0+\phi_i$ with $i \in I$ constitute a \emph{sunflower} with \emph{center} $\phi_0$ if the supports of the forms $\phi_i$ with $i \in I$ are pairwise disjoint. \end{definition}

The assumption of Theorem \ref{Main result} is that any two distinct forms differ by at least some number $M$ of coefficients. Our intermediate statements in stages (i) and (ii) will be more convenient to use and to prove with a modified assumption: that after subtracting the centre of the sunflower or of the ball to all forms, the resulting forms all have support size bounded below by some integer $r$. Both assumptions are of course closely related: the triangle inequality on the support distance immediately shows that for $r = \lceil M/2 \rceil$ the first implies the second except for possibly one of the forms.

\subsection{Exponentially small densities from the sunflower structure}

The most basic structure on mod-$p$ forms that we will use to ensure that a set of points of $S^n$ satisfying many conditions involving these forms has small density - exponentially small in the number of forms for now, although not for long - is a sunflower.

\begin{proposition} \label{sunflower case} Let $p$ be a prime and let $k \ge 1$ be an integer. Suppose that $E_1, \dots, E_k$ are strict subsets of $\F_p$ and that $\phi_0, \phi_1, \dots, \phi_k: \F_p^n \to \F_p$ are linear forms such that $\phi_0 + \phi_1, \dots, \phi_0 + \phi_k$ constitute a sunflower with center $\phi_0$ and $\phi_1, \dots, \phi_k$ each have support size at least $\max_{1 \le i \le k} L(S,E_i)$. Then the set \[\{x \in S^n: (\phi_0+\phi_1)(x) \in E_1, \dots, (\phi_0+\phi_k)(x) \in E_k \}\] has density at most $p(1-\b(p,S))^k$ inside $S^n$.

\end{proposition}

\begin{proof}

For every $y \in \F_p$ the probability \[\P_{x \in S^n}((\phi_0+\phi_1)(x) \in E_1, \dots, (\phi_0+\phi_k)(x) \in E_k, \phi_0(x) = y)\] can be rewritten as \[ \P_{x \in S^n}(\phi_1(x) \in E_1 - \{y\}, \dots, \phi_k(x) \in E_k - \{y\}, \phi_0(x) = y), \] which in turn is at most \[ \P_{x \in S^n}(\phi_1(x) \in E_1 - \{y\}, \dots, \phi_k(x) \in E_k - \{y\}).\] Because the supports of $\phi_1, \dots, \phi_k$ are pairwise disjoint, the last expression is equal to the product \[\P_{x \in S^n}(\phi_1(x) \in E_1 - \{y\}) \dots \P_{x \in S^n}(\phi_k(x) \in E_k - \{y\}).\] The sets $E_1 - \{y\}, \dots, E_k - \{y\}$ have the same sizes as $E_1, \dots, E_k$ respectively, and every term of the product is strictly less than 1 by definition of $L(S,E_1), \dots, L(S,E_k)$, so is at most $1 - \b(p,S)$ by Proposition \ref{Lower bound on non-zero probabilities}. We obtain \[\P_{x \in S^n}((\phi_0+\phi_1)(x) \in E_1, \dots, (\phi_0+\phi_k)(x) \in E_k, \phi_0(x) = y) \le (1 - \b(p,S))^k\] for every $y \in \F_p$ and conclude by the law of total probability. \end{proof}

\subsection{From sunflowers to small balls}

We next tackle the case where the assumption that the mod-$p$ forms constitute a sunflower is relaxed to the assumption that they are contained in a ball of bounded radius. This assumption together with the other assumption of Proposition \ref{sunflower case} say that the supports of $\phi_1, \dots, \phi_k$ are bounded both below and above. The main proof step is that such a family that has many forms necessarily contains a sunflower with many forms. The argument is quite standard: it is a variant of that establishing the Erd\H os-Rado sunflower theorem (proved in \cite{Erdos and Rado}).

\begin{proposition}\label{sunflowers from balls}

Let $p$ be a prime, let $t \ge 1, k \ge 1, r \ge 0$ be integers. Assume that $\phi_0+\phi_1, \dots, \phi_0+\phi_k: \F_p^n \to \F_p$ are linear forms and that $\phi_1, \dots, \phi_k$ have supports at most $r$. If $k \ge p^r r! t^r$ then we can find $I \subset [k]$ with $|I| \ge t$ such that the linear forms $\phi_0+\phi_i$ with $i \in I$ constitute a sunflower with center $\phi_0$.

\end{proposition}

\begin{proof}
We proceed by induction on $r$. If $r=0$ then the result is immediate. Suppose that for some $r \ge 1$ the result holds for $r-1$. Let $M \subset [k]$ be a maximal set such that the linear forms $\phi_i$ with $i \in M$ have pairwise disjoint supports. If $|M| \ge t$ then we take $I=M$ and we are done. Otherwise, for every $i \in [k]$ the support of $\phi_i$ intersects the union \[\cup_{j \in M} Z(\phi_j)\] of supports, which has size at most $rt$, in at least one element. By the pigeonhole principle we can find $J \subset [k]$ with size at least $k / (p-1)rt$ such that all forms $\phi_i$ with $i \in J$ have a common element $z$ in their support, and furthermore have the same coefficient $u \in \F_p$ at that element of the support. The forms $\phi_i - u x_z$ with $i \in J$ then all have supports with size at most $r-1$. We then conclude by the inductive hypothesis. \end{proof}

Our result in the case of forms contained in balls of bounded radius then follows from applying Proposition \ref{sunflowers from balls} and Proposition \ref{sunflower case} successively.

\begin{corollary} \label{corollary after balls case}

Let $p$ be a prime, let $k \ge 1$ and $r \ge 0$ be integers. If $\phi_0, \phi_1, \dots, \phi_k: \F_p^n \to \F_p$ are linear forms, and $E_1, \dots, E_k$ are strict subsets of $\F_p$ satisfying \[\max_{1 \le i \le k} L(S,E_i) \le |Z(\phi_i)| \le r\] for every $i \in [k]$ then \[\P_{x \in S^n}((\phi_0+\phi_1)(x) \in E_1, \dots, (\phi_0+\phi_k)(x) \in E_k) \le p(1 - \b(p,S))^{(k/p^r r!)^{1/r}}.\] \end{corollary}

\subsection{Finishing the proof}

We are now ready to finish the proof of Theorem \ref{Main result}. Given an arbitrary family of mod-$p$ forms, either this family contains a large family of forms that jointly is approximately equidistributed, in which case we are done, or it does not, and we then reduce to the case of forms contained in a ball with bounded radius.

\begin{proof}[Proof of Theorem \ref{Main result}]

Let $\e>0$. We distinguish two cases. If we can find a set $I \subset [k]$ with size $u = \lceil p \log (2 \e^{-1}) \rceil$ satisfying \[|Z(\sum_{i \in I} a_i \phi_i)| \ge 5 p^4 \log (2\e^{-1})\] for every $a \in \F_p^I \setminus \{0\}$ then by Proposition \ref{equidistribution for k-tuples of mod-p forms} we have \[\P_{x \in S^n}(\forall i \in I, \phi_i(x) = y_i) \le 2p^{-u} \] for every $y \in \F_p^I$, and hence \[\P_{x \in S^n}(\forall i \in I, \phi_i(x) \in E_i) \le 2 (1-p^{-1})^u \le \e.\]

If on the other hand we cannot find a set $I$ as above, then we can partition $[k]$ into at most $p^{u}$ sets such that all forms $\phi_0 + \phi_i$ with index $i$ in a given set are contained in some ball of radius at most $r = 5 p^4 \log (2\e^{-1})$. By the pigeonhole principle we can find such a set containing at least $k/p^{u}$ forms. The assumption \eqref{assumption of main result} together with the triangle inequality show that all but at most one of these forms is at a distance at least $\max_{1 \le i \le k} L(S,E_i)$ from the center of the ball. Using that \[1-\b(p,S) \le \exp(-\b(p,S)),\] Corollary \ref{corollary after balls case} then provides the upper bound \[\P_{x \in S^n}(\forall i \in I, \phi_i(x) \in E_i) \le p\exp \left(-\b(p,S) ((k-p^u) / p^u p^r r!)^{1/r} \right). \] For the previous right-hand side to be at most $\e$, it suffices that \[k \le 2 p^u p^r r! (\b(p,S)^{-1} \log (p\e^{-1}))^r.\] Plugging in the values of $u$ and $r$ in terms of $\e$, that becomes \begin{equation} k \le 2 p^{p\log(2\e^{-1})+1} p^{5 p^4 \log (2\e^{-1})} (5 p^4 \log (2\e^{-1}))! (\b(p,S)^{-1} \log (p\e^{-1}))^{5 p^4 \log (2\e^{-1})}. \label{four terms} \end{equation} As $\e$ tends to $0$, the right-hand side (the dominant terms of which are the last two) grows as $\Omega(\e^{-c(p) \log \log \e^{-1}})$ for some $c(p)>0$. \end{proof}

\section{Open questions}\label{Section: Open questions}

Let us finish by discussing a few remaining questions left open by the present results and proofs.

\subsection{Quantitative bounds}

Example \ref{Example 4} shows in particular that the bounds obtained in the conclusion of Theorem \ref{Main result} (and Corollary \ref{Corollary of main result}) cannot be improved to anything better than inverse power bounds in $k$. The bounds from Theorem \ref{Main result} and Corollary \ref{Corollary of main result} are therefore not too far from the optimal bounds, which leads immediately to our first question.

\begin{question}

What are the optimal bounds in Theorem \ref{Main result} and Corollary \ref{Corollary of main result} ? In particular, can they be taken to be a negative power of $k$ ?

\end{question}

To obtain such bounds, it would suffice to make the last two terms of \eqref{four terms} each have a power dependence in $\e^{-1}$, as the first two already do. The third term would be replaced by a power of $\e^{-1}$ assuming the following slight variant of the notorious sunflower conjecture, which would then be used instead of Proposition \ref{sunflowers from balls} and which could be of independent interest.  (The sunflower conjecture itself was first formulated by Erd\H os and Rado in \cite{Erdos and Rado}, with recent celebrated progress by Alweiss, Lovett, Wu and Zhang in \cite{Alweiss Lovett Wu Zhang}.)

\begin{conjecture}

Let $p$ be a (not necessarily prime) integer, and let $r \ge 1$. Then there exist $c_1(p), c_2(p) >0$ such that if $k,t \ge 1$ are integers satisfying $k \ge c_1(p) t^{c_2(p)}$ and $v_1, \dots, v_k$ are elements of $\Z_p^n$ each with at most $r$ non-zero coordinates, then we can find a subset $I \subset [k]$ with size $t$ and pairwise disjoint subsets $X_i \subset [n]$ with $i \in \{0\} \cup I$ such that for every $i \in I$ the set of non-zero coordinates of $v_i$ is $X_0 \cup X_i$, and the restrictions of $v_1, \dots, v_k$ to their coordinates in $X_0$ coincide.

\end{conjecture}

The last term of \eqref{four terms} presents, however, another difficulty on its own: fundamentally the extra $\log \log \e^{-1}$ term in the exponent is due to the fact that the parameter $u$ in that proof grows with $\e^{-1}$, and this dependence in turn appears to be quite basic and difficult to get around. This suggests that a proof which manages to obtain a power bound in Theorem \ref{Main result} would involve substantial extra ideas compared to the current proof.

\subsection{Remaining qualitative questions}

In another direction we can try to understand better which sets of conditions \[\{(\phi_i, E_i): i \in I\}\] lead to sets \[\{x \in S^n: \forall i \in I, \phi_i(x) \in E_i\}\] that are dense inside $S^n$ and which do not. This question is perhaps best formulated in an infinitary manner. We consider a family of linear forms from $\F_p^{\N}$ to $\F_p$ with finite support, that is, a family of linear forms of the type \[x \mapsto a_1 x_1 + \dots + a_n x_n\] for some finite $n$ (that depends on the form) and some coefficients $a_1, \dots, a_n \in \F_p$. If $k \ge 1$ is an integer, $\phi_1, \dots, \phi_k$ are such linear forms, and $E_1, \dots, E_k$ are subsets of $\F_p$, then we say that the \emph{density of the satisfying set} of the conditions $\phi_1(x) \in E_1, \dots, \phi_k(x) \in E_k$ is the density of the set \[\{x \in S^N: \phi_1^N(x) \in E_1, \dots, \phi_k^N(x) \in E_k\}\] inside $S^N$, where $N$ is a value of $n$ that works for all linear forms $\phi_1, \dots, \phi_k$, and the forms $\phi_1^N, \dots, \phi_k^N: \F_p^N \to \F_p$ are defined by \[\phi_i^N(x_1, \dots, x_N) = \phi_i(x)\] for every $i \in [k]$ and every $x \in \F_p^{\N}$. We say that a family of conditions $\{(\phi_i, E_i): i \in I\}$ with $\phi_i: \F_p^{\N} \to \F_p$ and $E_i \subset \F_p$ for every $i \in I$ is \emph{density-reducing} if there exists a function $F: \N \to \N$ tending to $0$ at infinity such that whenever $I’ \subset I$ is a finite set, then the density of the satisfying set of the conditions $\{(\phi_i, E_i): i \in I’\}$ is at most $F(|I’|)$. 

For instance, the family of conditions \[\{x_{i} = 0: i \ge 1 \}\] is density-reducing for $S = \{0,1\}$, since we may take $F(k) = 2^{-k}$, whereas the family of conditions \[\{x_1 + x_{i} \in \{0,1\}: i \ge 2\}\] is not density-reducing for $S = \{0,1\}$, as discussed in Example \ref{Example 1}.

With this definition, the qualitative side of the present paper’s contribution can be reformulated in terms of properties which ensure that a family of conditions $\{(\phi_i, E_i): i \in I\}$ with $E_i \neq \F_p$ is density-reducing. While Proposition \ref{equidistribution for k-tuples of mod-p forms} showed that it suffices that for every integer $r$ the set of forms $\{\phi_i: i \in I\}$ cannot be contained in finitely many balls with radius $r$, we now know by our Theorem \ref{Main result} that it suffices to guarantee that for the single value $r=2p-1$. A definitive characterisation remains open.

\begin{question}

Let $p$ be a prime, and let $S$ be a non-empty subset of $\F_p$. Can we characterise families of conditions $\{(\phi_i, E_i): i \in I\}$ that are density-reducing in $S^n$, where $\phi_i: \F_p^{\N} \to \F_p$ are linear forms with finite support and $E_i$ are strict subsets of $\F_p$ ?

\end{question}

In the special case $S=\F_p$, this question has a simple answer: the family is density-reducing if any only if we are in one of the following two cases: (i) the linear subspace $\langle \phi_i: i \in I \rangle$ has infinite dimension, or (ii) (the family involves finitely many pairwise distinct conditions and) and no $x \in S^n$ satisfies them, that is, informally speaking, the conditions are incompatible.

On the broader topic of understanding the behaviour of mod-$p$ forms on $\{0,1\}^n$ and on $S^n$ we may want to better understand how conditions may be deduced formally from one another. Given a non-empty subset $S$ of $\F_p$, if $\phi_1$, $\phi_2: \F_p^n \rightarrow \F_p$ are two linear forms, and $E_1$, $E_2$ are two subsets of $\F_p$, then $\{(\phi_1 + \phi_2, E_1 + E_2)\}$ is implied by $\{(\phi_1, E_1),(\phi_2, E_2)\}$ in the sense of the inclusion \[ \{x \in S^n: \phi_1(x) \in E_1, \phi_2(x) \in E_2\} \subset \{x \in S^n: (\phi_1 + \phi_2)(x) \in E_1 + E_2\} \]
but it is also the case that we can deduce $(\phi_1 + \phi_2)(x) \in E_1 + E_2$ formally from $\phi_1(x) \in E_1$ and $\phi_2(x) \in E_2$.

\begin{question} Given a non-empty subset $S$ of $\F_p$ and a family $\{(\phi_i, E_i): i \in I \}$ of $\bmod$-$p$ conditions, is there a class of formal operations which is simple to describe and such that whenever $(\phi, E)$ is a $\bmod$-$p$ condition satisfying \[\{x \in S^n: \forall i \in I, \phi_i(x) \in E_i\} \subset \{x \in S^n: \phi(x) \in E\},\] then $(\phi,E)$ can be formally deduced from the conditions $(\phi_i, E_i)$ with $i \in I$ using these operations? \end{question}

\end{document}